\documentclass[a4paper,12pt]{amsart}
\usepackage{amsmath,amsfonts,amssymb,amsthm}
\usepackage{latexsym,graphicx}
\usepackage{xypic}
\usepackage[matrix,arrow,ps,color,line,curve,frame]{xy}
\usepackage{tikz}
\usepackage{cancel}

\newtheorem{proposition}{\sc Proposition}[section]
\newtheorem{lemma}[proposition]{\sc Lemma}

\newtheorem{theorem}[proposition]{\sc Theorem}
\newtheorem{definition}[proposition]{\sc Definition}
\theoremstyle{definition}

\theoremstyle{remark}

\newcommand{\id}{\operatorname{id}}

\renewcommand{\phi}{\varphi}
\renewcommand{\epsilon}{\varepsilon}

\renewcommand{\[}{\begin{equation}}
\renewcommand{\]}{\end{equation}}

\def\id{{\rm id}}

\newlength{\bibitemsep}
\newlength{\bibparskip}
\let\oldthebibliography\thebibliography
\renewcommand\thebibliography[1]{%
  \oldthebibliography{#1}%
  \setlength{\parskip}{\bibitemsep}%
  \setlength{\itemsep}{\bibparskip}%
}

\begin{document}
\numberwithin{proposition}{section}
\numberwithin{equation}{section}
\baselineskip=13pt
\author{Piotr~M.~Hajac}
\address{Instytut Matematyczny, Polska Akademia Nauk, ul.~\'Sniadeckich 8, Warszawa, 00--656 Poland} 
\email{pmh@impan.pl}
\author{Oskar~M.~Stachowiak}
\address{Wydział Fizyki, Uniwersytet Warszawski, ul.~Pasteura 5, Warszawa, 02-093 Poland}
\email{o.stachowiak@uw.edu.pl}
\title[Counting paths in directed graphs]{{\Large Counting paths in directed graphs}}
\maketitle
\vspace*{-10mm}\begin{abstract}
We consider the class of directed graphs with $N\geq 1$ edges and without loops shorter than~$k\geq1$. 
Using the concept of a labelled graph, we determine graphs from this class that maximize the number of all paths of length~$k$.
Then we show an $R$-labelled version of this result for semirings $R$ contained in the semiring of non-negative real numbers and containing
the semiring of non-negative rational numbers. We end by posing a related open problem
concerning the maximal dimension of the path algebra of a connected acyclic directed graph with $N\geq1$
edges.
\end{abstract}


\section{Introduction}
\noindent 
Graph theory is considered one of the oldest and most accessible branches of combinatorics and has numerous 
natural connections to other areas of mathematics. 
In particular, directed graphs, or quivers, are fundamental tools in representation 
theory \cite{ass06} as well as in noncommutative geometry \cite{pr06}
and topology~\cite{hrt20,cht21}. 
In this paper, we focus entirely on the combinatorics of directed graphs: 
optimization and counting problems are common throughout combinatorics, and our paper solves one of them
(cf.~\cite{ssz23}).

Our goal is to extend the known Theorem~\ref{oldthm} counting fixed-positive-length paths in acyclic directed graphs
to the realm of directed graphs admitting loops. 
Without any assumptions on which loops are admissible, the counting
problem of Theorem~\ref{oldthm} trivializes as the maximal number of paths of length $k$ in a directed graph with 
$N$ edges is $N^k$, and the directed graph realizing
this number is the directed graph with one vertex and $N$ edges (a~Hawaiian earring). 
To overcome this problem, we propose that to count paths of length $k$ we should
only allow loops of length at least~$k$. This leads to our main result, which is Theorem~\ref{main}.

In the next section, we recall basic definitions and results needed for our paper. 
Some of them are standard, some of them come from~\cite{c-a20}, and there is the
inspirational Theorem~\ref{oldthm}. Then, in the subsequent section, first we  adapt
the results of Alexandru Chirvasitu
to the setting of paths of length $k$ in finite $R$-labelled directed graphs whose loops are of length at least~$k$. This allows us to prove our 
main theorem. We follow with a bonus theorem, which is an $R$-labelled version of Theorem~\ref{main}. Finally, we apply the $R$-labelling to find an 
exponential bound on the growth with the number of edges of the amount 
of all positive-length paths in an acyclic directed graph in terms of the number of edges. As we consider only directed graphs, in the remaining 
part of the paper we will often refer to directed graphs as graphs.

\section{Preliminaries}
\begin{definition}
Let $R$ be a commutative semiring. An $R$-labelled directed graph 
\[\nonumber
E:=(E^0,E^1,s,t,l)
\] 
is a quintuple consisting of
\begin{enumerate}
\item
 the set of vertices $E^0$ and the set of edges $E^1$,
\item 
the source and target maps $s,t\colon E^1\rightarrow E^0$
assigning to each edge its source vertex and target vertex, respectively,
\item 
the label map 
$l:E^{1}\rightarrow R\setminus\{0\},$
\end{enumerate}
satisfying the uniqueness condition (no repeating edges)
\[\nonumber
\big(s(e)=s(f) \text{ and } t(e)=t(f)\big)\quad\Rightarrow\quad e=f.
\]
\end{definition}
\noindent 
For $R=\mathbb{N}$, we recover the usual definition of a  graph with finitely many edges between any two vertices by
 replacing an edge $e$ with label $n\in \mathbb{N}$ by $n$ 
edges all starting at $s(e)$ and ending at $t(e)$.

A \emph{finite path} in an $R$-labelled  graph $E$ is a vertex or a finite sequence $p_n:=(e_1,\ldots,e_n)$ of edges
satisfying 
\begin{equation}
t(e_1)=s(e_2),\quad t(e_2)=s(e_3),\quad \ldots,\quad t(e_{n-1})=s(e_n).
\end{equation}
The beginning and the end of a vertex is the vertex itself.
The beginning $s(p_n)$ of $p_n$
is $s(e_1)$ and the end $t(p_n)$ of $p_n$ is $t(e_n)$. If $s(p_n)=t(p_n)$, we call $p_n$ 
a \emph{loop}, or a \emph{cycle}. An $R$-labelled  graph without loops is called acyclic.
The length of a path is the length of the sequence. Every edge is a path of length~$1$, and
vertices are considered as finite paths of length~$0$. The set of all paths in $E$ of length $k$ is denoted by~$FP_{k}(E)$, and the set of all finite 
paths in $E$ by~$FP(E)$.
Infinite paths are sequences of edges satisfying $t(e)=s(f)$ whenever the edge $f$ follows the edge $e$, 
and going to infinity from the initial edge, 
arriving from infinity to the final edge, or being infinite in both directions. Undirected paths are paths in which the condition 
 $t(e)=s(f)$ whenever the edge $f$ follows the edge $e$ is replaced by the requirement that 
$\{s(e),t(e)\}\cap\{s(f),t(f)\}\neq\emptyset$ whenever the edge $f$ follows the edge~$e$. We say that an $R$-labelled 
 graph is \emph{connected} when
any two vertices are connected by a finite (in the sense as above) undirected path.

\begin{theorem}[\cite{ht19,c-a20}]\label{oldthm}
Let $E$ be an acyclic directed graph with $N\geq 1$ edges, and let 
\[\nonumber
1\leq k\leq N=:nk+r,\quad
0\leq r\leq k-1.
\]
Then there are at most
\[\nonumber
\boxed{
P^N_k:=(n+1)^rn^{k-r}
}
\]
different paths of length $k$, and the bound is optimal.
\end{theorem}
\noindent
Since the number $P^N_k$ corresponds to all possible combinations of dealing the deck of $N$ cards to $k$ players, we call it
the \emph{card-deck} number.

\begin{lemma}\label{repeat}
Let $E$ be an $R$-labelled directed graph. If any edge repeats itself in a path $(e_1,\ldots,e_k)$ in~$E$, or if $|\{t(e_i)\}_1^k|<k$,
then there is a loop of length less than $k$.
\end{lemma}

\begin{proof}
Let $p:=(e_1,\ldots,e_{k})$ be a path of length~$k$.  
If  $e_i=e_j$ for $i<j$, then
\[
s(e_i)=s(e_j)=t(e_{j-1}),
\]
so the path $p_{ij}:=(e_i,\ldots,e_{j-1})$ is a loop:
\[
s(p_{ij})=s(e_{i})=t(e_{j-1})=t(p_{ij}).
\]
The smallest possible $i$  is $1$ and the largest possible  $j$ is $k$, so
 the maximal length of the above loop is $k-1<k$.
Much in the same way, if  $t(e_i)=t(e_j)$ for $i<j$, then $(e_{i+1},\ldots,e_j)$ is a loop of length $j-i\leq k-1$.
\end{proof}

\begin{theorem}\label{permutation}
Let $E$ be an $R$-labelled directed graph. If there exists a path p that is finite, or infinite in one direction,
and whose edges can be non-trivially rearranged (permuted) into a path, then there exists a loop in $E$.
\end{theorem}

\begin{proof}
For starters, note that a path infinite in one direction in $E$ is a path infinite in the opposite direction in the
opposite graph  $E^{op}$, i.e.\ the graph in which the source and the target maps are interchanged. 
Furthermore, a non-trivial permutation of edges in a path in $E$ is tantamount to
a non-trivial permutation of edges in this path considered as a path in $E^{op}$. 
Also, any loop in $E$ is  a loop in~$E^{op}$ and vice versa.
Assume now that having a path $(e_i)_{i\in\mathbb{N}}$ in $E$ that can be non-trivially permuted into
a path in $E$ implies that there is a loop in~$E$. Then having a path $(f_j)_{j\in -\mathbb{N}}$ in $E^{op}$ 
that can be non-trivially permuted into
a path in $E^{op}$ implies that there is a loop in~$E^{op}$. Therefore, as for any graph $F$ we have
$F=(F^{op})^{op}$, we can assume without the loss of generality that our path is finite or has edges 
indexed by~$\mathbb{N}$.

Let $S$ be a subset of $\mathbb{N}$ containing at least two elements, and let $\sigma: S\to S$ be a bijection that is not the identity.
Since $\sigma\neq\id$, there exist the smallest $j\in S$ such that $\sigma(j)\neq j$. As $\sigma$ is bijective, $\sigma(j)>j$. 
Indeed, if $j$ is the smallest
element of $S$, we are done. If there is $i<j$, then $\sigma(j)\neq\sigma(i)=i$, so $\sigma(j)>j$. 
Furthermore,   $\sigma^{-1}(j)\neq j$. If $\sigma^{-1}(j)< j$, then
we get a contradiction: $j=\sigma(\sigma^{-1}(j))=\sigma^{-1}(j)< j$. Therefore, also  $\sigma^{-1}(j)> j$.
 
 Next, let $p:=(e_1,...,e_n,...)$ or $p:=(e_1,...,e_n)$ with $n\geq 2$. Now,  let $S:=\mathbb{N}$ or $S:=\{1,...,n\}$, 
respectively. Suppose  that 
 $p_\sigma:=(e_{\sigma(1)},...,e_{\sigma(n)},...)$ or $p_\sigma:=(e_{\sigma(1)},...,e_{\sigma(n)})$ is again a 
path for a bijection $\sigma\neq\id$.
 Then $(e_j,\ldots,e_{\sigma^{-1}(j)})$ is a subpath of~$p$, so $(e_{\sigma(j)},...,e_j)$ is a subpath of 
$p_\sigma$. If  $\sigma(j)=j+1$, then the path $(e_{\sigma(j)},...,e_{j})=(e_{j+1},...,e_j)$
 is already a loop. If $\sigma(j)>j+1$, then, following this path with the
 path $(e_{j+1},\ldots,e_{\sigma(j)-1})$, we obtain a loop: 
$(e_{\sigma(j)},...,e_j,e_{j+1},\ldots,e_{\sigma(j)-1})$.
\end{proof}

\begin{definition}[\cite{c-a20}]
Let $E$ be a finite $R$-labelled directed graph.
The elements of $R$
\[\nonumber
ct(E):=\prod_{e\in E^{1}}l(e),\quad wt(E):=\sum_{e\in E^{1}}l(e),
\]
are called the content and the weight of $E$, respectively.
Furthermore, for $S\subseteq E^{1}$, the $S$-exclusive content of $E$ is
\[\nonumber
ct_{\cancel{S}}(E):=\prod_{e\in E^{1}\setminus S}l(e). 
\]
\end{definition}

\begin{definition}[\cite{c-a20}]
Let $E$ be a finite  $R$-labelled directed graph and $\Gamma$
 be a fixed finite directed graph without repeated edges. We define
\[\nonumber
ct^{\Gamma}_{R}(E):=\sum_{\gamma}ct(\gamma),
\]
with $\gamma$ ranging over the $R$-labelled subgraphs of $E$ isomorphic to $\Gamma$ as directed graphs.
\end{definition}

\begin{definition}
Let $E$ be a finite  $R$-labelled directed graph, and $k\in\mathbb{N}\setminus\{0\}$. We define
\[\label{pathct}
ct^{k}_{R}(E):=\!\!\!\sum_{p\in FP_{k}(E)}\!\!\!ct(p),\;\;\; ct\big((e_{1},
\ldots,e_{k})\big):=\prod_{i=1}^{k}l(e_{i}),
\;\;\; ct_{\cancel {e_j}}\big((e_{1},\ldots,e_{k})\big):=\!\!\!\prod_{i\in \{1,\dots k\}\setminus\{j\}}\!\!\!l(e_{i}).
\]
\end{definition}
\noindent
Note that even if  a finite  $R$-labelled  graph $E$ is not acyclic, so $FP(E)$ is infinite, $FP_{k}(E)$ is still 
finite for any $k$, 
so the sum in \eqref{pathct}
is finite.

Furthermore, to maximize $ct^{\Gamma}_{R}(E)$ for finite $R$-labelled graphs $E$ with a fixed weight $N$ 
as in \cite[Theorem~3.2]{c-a20} (which should be more precisely formulated than it is), we need the 
following fundemental concept:
\begin{definition}[\cite{c-a20}]\label{Chirvasitu}
Let $\Gamma$ be a fixed finite directed graph without repeated edges and containing at least one edge,
and let $E$ be an $R$-labelled directed graph. We say that $E$ has the \emph{Chirvasitu property} for $\Gamma$
when any two edges of $E$ belong to some common embedded copy of $\Gamma\subseteq E$.
\end{definition}

 As we are interested in maximizing $ct^{k}_{R}(E)$ rather than $ct^{\Gamma}_{R}(E)$,
let us first
point out that, if an $R$-labelled  graph is not acyclic, then paths are not necessarily subgraphs. Indeed, in the graph
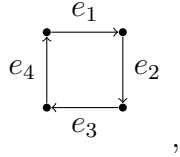
\begin{center}
\begin{tikzpicture}[auto,swap]
\tikzstyle{vertex}=[circle,fill=black,minimum size=3pt,inner sep=0pt]
\tikzstyle{edge}=[draw,->]
\tikzset{every loop/.style={min distance=20mm,in=130,out=50,looseness=50}}
    \node[vertex] (1) at (-1,0) {};
   \node[vertex] (2) at (0,0) {};
    \node[vertex] (3) at (0,-1) {};
    \node[vertex] (4) at (-1,-1) {};

    \path (1) edge [edge,above] node {$e_1$} (2) ;

    \path (2) edge [edge,right] node {$e_2$} (3);
    \path (3) edge [edge,below] node {$e_{3}$} (4);
  \path (4) edge [edge,left] node {$e_{4}$} (1);
\end{tikzpicture},
\end{center}
the two different paths $(e_1,\ldots,e_4,e_1,\ldots,e_4)$ and $(e_1,\ldots,e_4)$ have the same underlying subgraph.

However, due to Theorem~\ref{permutation}, any finite path in an acyclic $R$-labelled  graph
 can be identified with a subgraph because then the order of edges is uniquely fixed.
In this setting, we have the following result (which is a special case of Lemma~\ref{lemma2}):
\begin{lemma}\label{aclem}
 Let $E$ be a finite acyclic $R$-labelled directed graph such that 
\[\nonumber
s(E^1)\cup t(E^1)=E^0\neq \emptyset
\] 
and $p$ be a path in $E$ of length~$k\geq 1$. Then the only such graphs  satisfying the Chirvasitu property 
for $p$ with its labelling
omitted are paths $(e_{1},\ldots,e_{k})$.
\end{lemma}

\section{Counting paths in non-acyclic graphs}
\noindent
First, we need to prove \cite[Theorem~3.2]{c-a20} and Lemma~\ref{aclem} in the desired setting. Note that the proof of our path version
of \cite[Theorem~3.2]{c-a20} enjoys basically the same proof as the original theorem.
\begin{lemma}\label{lemma1}
Let $R$ be a  subsemiring of $\mathbb{R}_{\geq 0}$, $N\in R$, and $FG^{N}_{R}$ be
the set of all finite $R$-labelled directed graphs of weight $N$.
Then, 
\[ \nonumber
\sup_{E\in FG^{N}_{R}}\: ct^{k}_{R}(E),\quad\text{for any}\quad k\in\mathbb{N}\setminus\{0\},
\]
is zero or is achieved over graphs $E$ with the following property (the Chirvasitu property for paths):
\begin{gather}
   \text{Any two edges of $E$ belong to the same path of length $k$, i.e.}\label{pathac}\\
    \text{$e,f\in E^{1}\;\Longrightarrow\;\exists$ a path $p$ of length $k$ such that both edges $e$ and $f$ belong to $p$.}
 \nonumber \end{gather}
\end{lemma}
\begin{proof}
The case $N=0$ is trivial, so we assume that $N>0$.  Next, the supremum is zero if and only if there are no graphs
with paths of length~$k$, so we can further 
assume that $E$ has at least one path of length~$k$.
Now, let $e,f\in E^{1}$, 
and let us suppose that they do not belong to the same path of length~$k$. 
Next, let $S_{e}$ and $S_{f}$ be the sets of all $k$-paths containing $e$ and $f$, respectively. Then $S_{e}\cap S_{f}=\emptyset$.
Furthermore, observe that
\begin{align}\nonumber
ct^{k}_{R}(E)&=\sum_{\alpha\in S_{e}}ct(\alpha)+\sum_{\beta\in S_{f}}ct(\beta)+\sum_{\gamma\notin S_{e}\cup S_{f}}ct(\gamma)
\\ &=
l(e)\sum_{\alpha\in S_{e}}ct_{\cancel{e}}(\alpha)+l(f)\sum_{\beta\in S_{f}}ct_{\cancel{f}}(\beta)+
\sum_{\gamma\notin S_{e}\cup S_{f}}ct(\gamma)
\end{align}
and assume, without  loss of generality,  that
\[
\sum_{\alpha\in S_{e}}ct_{\cancel{e}}(\alpha) \geq\sum_{\beta\in S_{f}}ct_{\cancel{f}}(\beta).
\]

Now we produce a graph $E’$ of the same weight $N$  by eliminating the edge $f$ and adding its label to the label of the edge $e$,
so that the label of $e$ is now $l(e)+l(f)$.
The $k$-content of  $E'$  is
\[
ct^{k}_{R}(E')=\big(l(e)+l(f)\big)\sum_{\alpha\in S_{e}}ct_{\cancel{e}}(\alpha)+\sum_{\gamma\notin S_{e}\cup S_{f}}ct(\gamma).
\]
Finally, we  compute:
\begin{align}
\nonumber 
ct^{k}_{R}(E')-ct^{k}_{R}(E)&=\big(l(e)+l(f)\big)\sum_{\alpha\in S_{e}}ct_{\cancel{e}}(\alpha) 
-l(e)\sum_{\alpha\in S_{e}}ct_{\cancel{e}}(\alpha) -l(f)\sum_{\beta\in S_{f}}ct_{\cancel{f}}(\beta)\\
&= l(f)\Big(\sum_{\alpha\in S_{e}}ct_{\cancel{e}}(\alpha)-\sum_{\beta\in S_{f}}ct_{\cancel{f}}(\beta)\Big)\geq 0.
\end{align}

Hence, whenever we have a finite $R$-labelled  graph with a pair of edges that do not belong to the same path of length $k$, 
we can always eliminate one of these two edges to 
create a new  finite $R$-labelled  graph with at least the same $k$-content. Since the number of edges is finite, 
we finally arrive at a finite $R$-labelled  graph satisfying~\eqref{pathac}.
\end{proof}

Also the following lemma is just a variation of \cite[Lemma 3.3]{c-a20} adapted to 
    our goals. 
\begin{lemma}\label{lemma2}
Let $k\in\mathbb{N}\setminus\{0\}$ and let $E$ be a finite $R$-labelled directed graph without loops shorter than~$k$
and such that 
\[\label{connect}
s(E^1)\cup t(E^1)=E^0\neq\emptyset.
\]
Then, the only such  graphs satisfying \eqref{pathac} 
(the Chirvasitu property for 
paths of length $k$) are path $(e_1,\ldots,e_k)$ directed graphs such 
that $|\{t(e_i)\}_{i=1}^{k}|=k$, or loop  $(f_1,\ldots,f_m)$ directed graphs such
 that $k\leq m\leq 2k-1$ and $|\{t(f_{i})\}_{1}^{m}|=m$.
 \end{lemma}
\begin{proof}
Suppose that $E$ contains one of the following subgraphs:
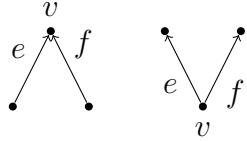
\begin{center}
\begin{tikzpicture}[auto,swap]
\tikzstyle{vertex}=[circle,fill=black,minimum size=3pt,inner sep=0pt]
\tikzstyle{edge}=[draw,->]
\tikzset{every loop/.style={min distance=20mm,in=130,out=50,looseness=50}}
    \node[vertex,label=above:$v$] (0) at (-1.5,0) {};
    \node[vertex] (2) at (-1,-1) {};
    \node[vertex] (3) at (-2,-1) {};

    \path (2) edge [edge,above] node {$\;\;\;f$} (0) ;
    \path (3) edge [edge,above] node {$e\;\;\;$} (0);

\tikzstyle{vertex}=[circle,fill=black,minimum size=3pt,inner sep=0pt]
\tikzstyle{edge}=[draw,->]
\tikzset{every loop/.style={min distance=20mm,in=130,out=50,looseness=50}}
    \node[vertex,label=below:$v$] (0) at (0.5,-1) {};
    \node[vertex] (2) at (1,0) {};
    \node[vertex] (3) at (0,0) {};

    \path (0) edge [edge,below] node {$\;\;\;f$} (2) ;
    \path (0) edge [edge,below] node {$e\;\;\;$} (3);
\end{tikzpicture}.
\end{center}
We call these subgraphs $\Lambda$ and $V$, respectively. For any of them, 
to have both $e$ and $f$ in the same path of length $k$ we must have a loop based at $v$ of length at most $k-1$,
which contradicts our assumption. Therefore, $E$ cannot contain such subgraphs.

Next, the combination of \eqref{pathac} with \eqref{connect} implies that $E$ must be a non-empty connected graph. Furthermore,
as $E^1\neq\emptyset$, by \eqref{pathac} $E$ must have at least one path of length~$k$. Hence, it must have at least $k$ edges 
and its paths of the form $(e_1,\ldots,e_k)$ must satisfy $|\{t(e_{i})\}_{1}^{k}|=k$ by 
Lemma~\ref{repeat} and the assumption that there are no loops shorter than~$k$. It follows that, if $|E^1|=k$, then $E$ is the graph
underlying a path $(e_1,\ldots,e_k)$, which is
either an open path of length $k$ without repeated edges or a loop of length $k$ satisfying  $|\{t(e_{i})\}_{1}^{k}|=k$.

If $|E^1|=m>k$, then our path $(e_1,\ldots,e_k)$ cannot be a loop as then there would be no way to attach any of the remaining $m-k$ edges
without creating a forbidden $\Lambda$ or $V$ subgraph. Hence,  $(e_1,\ldots,e_k)$ is an open path, and the only way to attach an edge to it
is at the beginning or the end. Now, the loose end of the attached edge can either remain loose or be identified with the opposite end of 
$(e_1,\ldots,e_k)$. Thus, either we have an open path of length $k+1$ without repeated edges or vertices, or a loop of length $k+1$
with $k+1$ different edges and $k+1$ different vertices. Next, if $m=k+1$, then only the latter case satisfies \eqref{pathac}. If $m>k+1$,
then only the former case allows us to attach one of the remaining $m-k-1$ edges. We repeat this reasoning inductively until we reach
$m=2k$. Then neither of the cases satisfies \eqref{pathac}, and the procedure halts exhausting all graphs claimed in the lemma.
\end{proof}

We are now ready to prove our main result:
\begin{theorem}\label{main}
Let $E$ be a directed graph with $N\geq 1$ edges, and let $1\leq k\leq N=:nk+r$,
$0\leq r\leq k-1$. 
Assume also that  there are no loops shorter than~$k$.
Then there are at most 
$kP^N_k$
different paths of length $k$, and the bound is optimal.
\end{theorem}
\begin{proof}
The claim is trivial for $k=1$, so we assume that $k\geq2$. Next, we put $R=\mathbb{N}$ in Lemma~\ref{lemma1} and  
Lemma~\ref{lemma2}. 
Furthermore, without the loss of generality,
we can assume that $E$ satisfies \eqref{connect} as {red}$|E^1|\geq 1$ and disconnected vertices have no bearing on the number of paths
of length~$k$. Now we can treat $E$ as a finite $\mathbb{N}$-labelled  graph satisfying~\eqref{connect}, without loops shorter
than~$k$, and of the weight~$N$. Thus, all assumptions of both lemmas are fulfilled, and we can conclude that graphs maximizing the
number of all paths of length $k$ are of the form indicated in Lemma~\ref{lemma2}. It is clear that an open $\mathbb{N}$-labelled
path of length $k$ has fewer paths of length $k$ than  an  $\mathbb{N}$-labelled loop of length~$k$. Now we will show that shrinking
$\mathbb{N}$-labelled loops of length $m$, with $k<m<2k$, increases the number of $k$-paths. To this end, we must have $k\geq 2$.

Let $(f_1,\ldots,f_m)$ be an $\mathbb{N}$-labelled loop as in Lemma~\ref{lemma2}.
 Without the loss of generality, we can assume that  $l(f_{1})$ is the lowest label: 
$l(f_{1}):=\min\{l(f_i)\}_1^m$. Now we can shrink our  graph~$E$
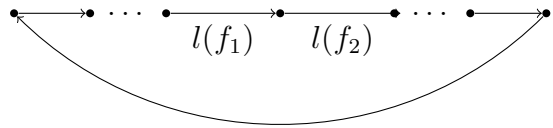
\begin{center}
\begin{tikzpicture}[auto,swap]
\tikzstyle{vertex}=[circle,fill=black,minimum size=3pt,inner sep=0pt]
\tikzstyle{edge}=[draw,->]
\tikzset{every loop/.style={min distance=20mm,in=130,out=50,looseness=50}}
    \node[vertex] (1) at (-1,0) {};
   \node[vertex] (2) at (0,0) {};
    \node (3) at (0.5,0) {\ldots};
    \node[vertex] (4) at (1,0) {};
    \node[vertex] (5) at (2.5,0) {};
    \node[vertex] (6) at (4,0) {};
    \node (6) at (4.5,0) {\ldots};
    \node[vertex] (7) at (5,0) {};
    \node[vertex] (8) at (6,0) {};
   
    \path (1) edge [edge]  (2) ;
    \path (4) edge [edge] node {$l(f_{1})$} (5);
    \path (5) edge [edge] node {$l(f_{2})$} (6);
    \path (7) edge [edge]  (8);
    \path (8) edge [edge,below,bend left=45] (1);
\end{tikzpicture}
\end{center}
by eliminating the edge $f_{1}$ and adding  its label to the label of the edge $f_{2}$. Thus, we obtain a graph~$E’$
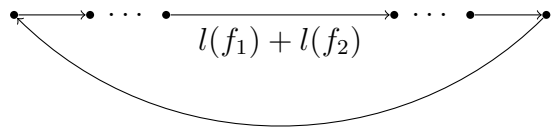
\begin{center}
\begin{tikzpicture}[auto,swap]
\tikzstyle{vertex}=[circle,fill=black,minimum size=3pt,inner sep=0pt]
\tikzstyle{edge}=[draw,->]
\tikzset{every loop/.style={min distance=20mm,in=130,out=50,looseness=50}}
    \node[vertex] (1) at (-1,0) {};
   \node[vertex] (2) at (0,0) {};
    \node (3) at (0.5,0) {\ldots};
    \node[vertex] (4) at (1,0) {};
    \node[vertex] (5) at (4,0) {};
    \node (6) at (4.5,0) {\ldots};
    \node[vertex] (7) at (5,0) {};
    \node[vertex] (8) at (6,0) {};
   
    \path (1) edge [edge]   (2) ;
    \path (4) edge [edge] node {$l(f_{1})+l(f_{2})$} (5);
    \path (7) edge [edge]   (8);
    \path (8) edge [edge,below,bend left=45]  (1);
\end{tikzpicture}
\end{center}
without loops shorter than $k$ and with the same weight (number of edges) $N$, which is an $\mathbb{N}$-labelled loop shorter by one. 
It remains to show that $|FP_k(E')|>|FP_k(E)|$.

Let $S_{1}\subseteq FP_k(E)$ be the set of all $k$-paths containing $f_{1}$ and $S_2\subseteq FP_k(E)$ and
$S'_2\subseteq FP_k(E')$ be the sets of all $k$-paths containing~$f_{2}$. To shorten notation, let us put 
$\prod_{i=x}^yl(f_i)=:\Pi_x^y$ and $\Pi_x^y:=1$ when $y<x$.
Then
\begin{align}
ct_{\mathbb{N}}^k(E)&=\sum_{j=2}^{k+1}\Pi_{m-k+j}^{m}\Pi_{1}^{j-1} + \Pi_{2}^{k+1}
+\!\!\!\!\!\!\sum_{\gamma\in FP_k(E)\setminus (S_1\cup S_2)}\!\!\!\!\!\!ct(\gamma)
\nonumber\\ &=
l(f_1)\Pi_{m-k+2}^{m} + \sum_{j=3}^{k+1}\Pi_{m-k+j}^{m}\Pi_{1}^{j-1} + \Pi_{2}^{k+1}
+\!\!\!\!\!\!\sum_{\gamma\in FP_k(E)\setminus (S_1\cup S_2)}\!\!\!\!\!\!ct(\gamma)
\nonumber\\ &=
l(f_1)\Pi_{m-k+2}^{m} + \sum_{j=2}^{k}\Pi_{m-k+j+1}^{m}\Pi_{1}^{j} + \Pi_{2}^{k+1}
+\!\!\!\!\!\!\sum_{\gamma\in FP_k(E)\setminus (S_1\cup S_2)}\!\!\!\!\!\!ct(\gamma)
\end{align}
and
\begin{align}
ct_{\mathbb{N}}^k(E')&=\big(l(f_1)+l(f_2)\big)\sum_{j=2}^{k+1}\Pi_{m-k+j}^{m}\Pi_{3}^{j}
+\!\!\!\sum_{\gamma\in FP_k(E')\setminus S'_2}\!\!\!ct(\gamma)
\nonumber\\ &=
l(f_1)\sum_{j=2}^{k+1}\Pi_{m-k+j}^{m}\Pi_{3}^{j}+\sum_{j=2}^{k+1}\Pi_{m-k+j}^{m}\Pi_{2}^{j}
+\!\!\!\sum_{\gamma\in FP_k(E')\setminus S'_2}\!\!\!ct(\gamma)
\nonumber\\ &=
l(f_1)\Pi_{m-k+2}^{m} + l(f_1)\sum_{j=3}^{k+1}\Pi_{m-k+j}^{m}\Pi_{3}^{j}+\sum_{j=2}^{k}\Pi_{m-k+j}^{m}\Pi_{2}^{j}
 + \Pi_{2}^{k+1}
+\!\!\!\sum_{\gamma\in FP_k(E')\setminus S'_2}\!\!\!ct(\gamma).
\end{align}
Hence,
\begin{align}
ct_{\mathbb{N}}^k(E')-ct_{\mathbb{N}}^k(E)
&=
 l(f_1)\sum_{j=3}^{k+1}\Pi_{m-k+j}^{m}\Pi_{3}^{j}+\sum_{j=2}^{k}\Pi_{m-k+j}^{m}\Pi_{2}^{j}
-\sum_{j=2}^{k}\Pi_{m-k+j+1}^{m}\Pi_{1}^{j}
\nonumber\\ &= 
 l(f_1)\sum_{j=3}^{k+1}\Pi_{m-k+j}^{m}\Pi_{3}^{j}+\sum_{j=2}^{k}\Pi_{m-k+j+1}^{m}\big(l(f_{m-k+j})-l(f_{1})\big)\Pi_{2}^{j}
\nonumber\\ &> 0.
\end{align}
Here the last step follows from the minimality of~$l(f_1)$ and the fact that $k\geq 2$.

Now we can conclude that the number of paths of length $k$ is maximized by $\mathbb{N}$-labelled loops of length $k$ 
with $k$ different vertices. It is clear
that the number of $k$-paths in such a loop equals $k$ times the number of $k$-paths in an $\mathbb{N}$-labelled open path 
obtained by cutting the loop 
at any vertex. As the latter number is maximized by $P^N_k$ by Theorem~\ref{oldthm}, we infer that the maximal number of $k$-paths
in our case is $kP^N_k$, and this maximum is realized by $\mathbb{N}$-labelled loops obtained by indentifying the beginning and the end of an
 $\mathbb{N}$-labelled $k$-path maximizing the number of  $k$-paths in the acyclic case.
\end{proof}

Furthermore, observe that the proof of the above theorem does not depend on labels being natural numbers. Therefore, repeating the proof
for a semiring $R$ such that $\mathbb{Q}_{\geq0}\subseteq R \subseteq\mathbb{R}_{\geq0}$ and using \cite[Theorem~6]{c-a20}, we 
arrive at the following bonus result:
\begin{theorem}
Let $R$ be a  subsemiring of $\mathbb{R}_{\geq 0}$ containing $\mathbb{Q}_{\geq0}$, $N\in R\setminus\{0\}$, $k\in\mathbb{N}\setminus\{0\}$ and 
$k$-$FG^{N}_{R}$ be
the set of all finite $R$-labelled directed graphs of weight $N$ without loops shorter than~$k$.
Then
\[ \nonumber
\sup_{\text{$E\in k$-$FG^{N}_{R}$}}\: ct^{k}_{R}(E)=k\left(\frac{N}{k}\right)^k.
\]
The supremum is achieved by the $R$-labelled directed graph given by the loop of length $k$ with $k$ vertices whose all edges 
have the same label $\frac{N}{k}$. 
\end{theorem}

To end with, let us state the following open problem. We already know the maximal number of 
fixed-positive-length paths in acylic graphs with $N$ edges (Theorem~\ref{oldthm}). However, the 
question of what is the maximal number of all paths in a connected acyclic graph
with $N$ edges remains wide open. It is particularly interesting because it is 
equivalent to the question of what is the biggest finite dimension of a path algebra with $N$ grade-one
generators~\cite{ass06}. Alas, all we can 
prove thus far is:
\begin{proposition}
Let $E$ be an acyclic directed graph with $N\geq 1$ edges. Then the number of all positive-length
paths is not greater than $N ( \sqrt[e]{e})^{N}$.
\end{proposition}
\begin{proof}
As the arithmetic mean  is always greater than or equal to the geometric mean, first we observe that $(\frac{N}{k})^k\geq P^N_k$.
Combining now Theorem~\ref{oldthm} with elementary calculus to maximize the function
\begin{equation}
[1,N]\ni x\longmapsto x^{\frac{N}{x}}\in\mathbb{R}
\end{equation}
 yields the claim when substituting $x=\frac{N}{k}$.
\end{proof}

\section*{Acknowledgements}
\noindent This work is part of the project ``Applications of graph algebras and higher-rank graph algebras in 
noncommutative geometry'' partially supported by NCN grant UMO-2021/41/B/ST1/03387.

\bibliographystyle{acm}

\end{document}